\documentclass[10pt]{amsart}
\usepackage{geometry} 
\newtheorem{theorem}{Theorem}[section]
\newtheorem{corollary}[theorem]{Corollary}
\newtheorem{definition}{Definition}[section]
 
\newtheorem{proposition}[theorem]{Proposition}

\makeatletter
\numberwithin{equation}{section}
\makeatother
\def\R{\mathbb{R}}
\def\bx{\mathbf{x}}
\def\bX{\mathbf{X}}
\def\by{\mathbf{y}}

\def\bu{\mathbf{u}}
\def\bv{\mathbf{v}}
\def\bbf{\mathbf{f}}
\def\bU{\mathbf{U}}
\def\bA{\mathbf{A}}

\def\bB{\mathbf{B}}

\def\bp{\mathbf{p}}
\def\bP{\mathbf{P}}
\def\bQ{\mathbf{Q}}
\def\bR{\mathbf{R}}

\def\bH{\mathbf{H}}
\def\bK{\mathbf{K}}

\def\cero{\mathbf{0}}

\title[Optimal feedback control]{Optimal feedback control, linear first-order PDE systems, and obstacle problems}
\author{Pablo Pedregal}
\thanks{
Universidad de Castilla La Mancha, INEI, 
Campus de Ciudad Real (Spain). Research supported by
MTM2013-47053-P of the Mineco (Spain), by PEII-2014-010-P of the Conserjer\'\i a de Cultura (JCCM), and by grant GI20152919 of UCLM.
 e-mail:{\tt pablo.pedregal@uclm.es
}}

\begin{document}

\maketitle
   \begin{abstract}
   We introduce an alternative approach for the analysis and numerical approximation of the optimal feedback control mapping. It consists in looking at a typical optimal control problem in such a way that feasible controls are mappings depending both in time and space. In this way, the feedback form of the problem is built-in from the very beginning. Optimality conditions are derived for one such optimal mapping, which by construction is the optimal feedback mapping of the problem. In formulating optimality conditions, costates in feedback form are solutions of linear, first-order transport systems, while optimal descent directions are solutions of appropriate obstacle problems. 
We treat situations with no constraint-sets for control and state, as well as the more general case where a constraint-set is considered for the control variable. 
    \end{abstract}

\section{Introduction}
Consider the optimal control problem
$$
\hbox{Minimize in }\bu(s)\in \bK:\quad I(\bu)=\int_0^T F(\bx(s), \bu(s))\,ds+g(\bx(T))
$$
subject to
$$
\bx'(s)=\bbf(\bx(s), \bu(s))\hbox{ in }(0, T),\quad \bx(0)=\bx_0, \bx(s)\in\Omega
$$
where:
\begin{itemize}
\item $T>0$ is the time horizon considered;
\item $\Omega\subset\R^N$ is the feasible set for the state variable $\bx:(0, T)\to\Omega$;
\item $\bK\subset\R^m$ is the feasible set for the control variable $\bu:(0, T)\to \bK$;
\item $\bx_0\in\R^N$ is the vector determining the state of the system when we start to care about the control problem;
\item $F:\Omega\times \bK\to\R$ is the density for the cost functional, while $g:\Omega\to\R$ is the contribution depending on the final state;
\item $\bbf:\Omega\times \bK\to\R^N$ is the map providing the state equation that governs the dynamics of the system.
\end{itemize}
With all of these ingredients given to us, we care about the following map
$$
\bU(t, \by):[0, T]\times\Omega\to \bK
$$
defined as follows. For $\by\in\Omega$, and $t\in[0, T]$, consider the problem
$$
\hbox{Minimize in }\bu(t)\in \bK:\quad I(\bu)=\int_t^T F(\bx(s), \bu(s))\,ds+g(\bx(T))
$$
subject to
$$
\bx'(s)=\bbf(\bx(s), \bu(s))\hbox{ in }(t, T),\quad \bx(t)=\by, \bx(s)\in\Omega.
$$
Let us assume, to let the discussion move ahead, that there is a unique optimal solution for this optimal control problem for every $t\in[0, T]$, and $\by\in\Omega$. 
Suppose $\bu(s; t, \by)$, for $s\in[t, T]$, is such optimal solution. Then we take
$$
\bU(t, \by)=\bu(t; t, \by)
$$
for a.e. $t\in[0, T]$, $\by\in\Omega$. 

\begin{definition}
This map $\bU(t, \by)$ is called the optimal, feedback control of the problem. 
\end{definition}

The relevance of this map is recorded in the following statement, which is hardly in need of further justification. It establishes that all optimal pairs for our optimal control problem above are always related through $\bU$. 
\begin{proposition}\label{inicial}
Let $(\bx(t), \bu(t))$ be an optimal pair for the control problem. Then $\bu(t)=\bU(t, \bx(t))$.  
\end{proposition}

The whole point of feedback control is to be able to compute (approximate) this mapping $\bU$ beforehand, so that when we come to finding the optimal solution of the original problem we are ready to adjust to disturbances that may occur during real processes by measuring (part of) the state of the system, and adjusting the optimal control through the optimal feedback mapping $\bU$.  

The classical way of trying to calculate $\bU(t, \by)$ is by considering the Hamilton-Jacobi-Bellman equation for the value function $v(t, \bx)$
\begin{equation}\label{escalar}
v_t(t, \bx)+H(\nabla v(t, \bx), \bx)=0\hbox{ in }(0, T)\times\R^{N,}\footnote{$\nabla$ designates throughout the gradient only with respect to the spatial variable $\bx$.}
\end{equation}
together with the terminal time condition $v(T, \bx)=g(\bx)$ for all $\bx\in\R^N$. Here we are taking $\Omega=\R^N$. The value function $v(t, \bx)$, and  
the hamiltonian $H(\bp, \bx)$ are defined as usual
\begin{equation}\label{hamiltoniano}
H(\bp, \bx)=\min_{\bu\in \bK} \{F(\bx, \bu)+\bp\cdot \bbf(\bx, \bu)\},
\end{equation}
and $v(t, \bx)$ is the optimal value of the above problem determining $\bU(t, \bx)$; that is, if $v(t, \bx)$ is known, then the optimal, feedback map $\bU(t, \bx)$ is precisely the vector $\bu=\bU(t, \bx)$ where the minimum
$$
\min_{\bu\in \bK} \{F(\bx, \bu)+\nabla v(t, \bx)\cdot \bbf(\bx, \bu)\}
$$
is realized. This beautiful theory is by now well-established through viscosity solutions of (\ref{escalar}). See \cite{EvansB}, \cite{soner}, for instance. 

Alternatively, one can focus directly on the field $\bv(t, \bx)=\nabla v(t, \bx)$ instead of on $v(t, \bx)$. It is elementary to argue that, by formally differentiating  the Hamilton-Jacobi-Bellman equation above with respect to $\bx$,
\begin{equation}\label{vectorial}
\bv_t(t, \bx)+H_\bp(\bv(t, \bx), \bx)\nabla\bv(t, \bx)+H_\bx(\bv(t, \bx), \bx)=0\hbox{ in }(0, T)\times\R^N
\end{equation}
together with the terminal time condition 
$$
\bv(T, \bx)=\nabla g(\bx).
$$
If $\bv(t, \bx)$ is known, then, as before, the optimal, feedback map is obtained through the optimal solution of the problem
$$
\min_{\bu\in \bK} \{F(\bx, \bu)+\bv(t, \bx)\cdot \bbf(\bx, \bu)\}.
$$

Both approaches have advantages and disadvantages. This second perspective may seem more appealing for two reasons. The first one  is that the field we need in order to compute $\bU(t, \by)$ is $\bv(t, \by)$. Bearing in mind that what we compute or approximate with the Hamilton-Jacobi-Bellman equation is $v(t, \bx)$, and then we need to approximate its spatial gradient $\nabla v(t, \bx)$, it may look reasonable to deal with a problem which directly furnishes this field $\bv(t, \bx)=\nabla v(t, \bx)$. 
The second reason is more important. Both problems (\ref{escalar}), and (\ref{vectorial}), can be treated with the method of characteristics. However, this somehow leads us back in both cases to solving the underlying Hamilton ODE system. Approximating either $v(t, \bx)$ or $\bv(t, \bx)$ through the characteristic scheme is like computing $\bU(t, \by)$ directly. Therefore, from a practical point of view, one would decide that problem (\ref{escalar}) or (\ref{vectorial}) for which numerical methods are better developed, or better known. Put it in this way, the second possibility may seem more attractive because of the semi-linear nature of (\ref{vectorial}) versus the fully-nonlinear equation (\ref{escalar}). However, (\ref{vectorial}) is a system while (\ref{escalar}) is a single equation. 
 
The truth is that from a practical viewpoint, either of the two procedures is not easy to implement, even for low dimension $N$, as hyperbolic first-order PDE or systems are delicate, and even more so is its numerical implementation (check \cite{beelertranbanks}). In addition, setting up (\ref{escalar}) or (\ref{vectorial}), require to have an explicit form of the hamiltonian $H(\bp, \bx)$ which involves to go through the minimization calculation with respect to the control variable $\bu$. Even in simple, academic examples, when a restriction set $\bK$ should be respected, the hamiltonian may be discontinuous so that the mathematical analysis of problems (\ref{escalar}) and (\ref{vectorial}) is far from straightforward. 
These practical difficulties has stirred certain interest in finding other ways to treat and approximate optimal feedback control. See \cite{bourfli}, \cite{edgoh}, \cite{rossekflego}. 

If all these difficulties could be resolved somehow, there is still the ``curse of dimension" issue which is a major barrier for all approaches including the one we describe here. It relates to the fact that differential problems cannot be solved or approximated in space $\R^N$ of high dimension $N$. Even dimension $N=3$ is pretty demanding. See however \cite{darbonosher}. 

Our motivation is the following.
\begin{quote}
Since according to Proposition \ref{inicial}, the optimal control $\bu(t)$ for the initial problem can always be represented in the form $\bu(t)=\bU(t, \bx(t))$, it may be worthwhile to look at the initial optimal control problem in the form
$$
\hbox{Minimize in }\bu(t, \bx)\in\bK:\quad 
\int_\Omega\int_0^T\int_t^TF(\bx(s), \bu(s, \bx(s)))\,ds\,dt\,d\by+g(\bx(T))
$$
subject to
$$
\bx'(s)=\bbf(\bx(s), \bu(s, \bx(s)))\hbox{ in }(t, T),\quad \bx(t)=\by, \bx(s)\in\Omega,
$$
for each pair $(t, \by)\in[0, T]\times\Omega$. 
\end{quote}

Feasible $\bu(t, \bx)$'s for this optimization problem are all possible feedback laws for the system. 
In this general format, $\bu(t, \bx)$ can be taken to be measurable in $t$ and continuous in $\bx$. Associated states $\bx(s):[t, T]\to\Omega$ will be absolutely continuous, point wise solutions of the state ODE system
$$
\bx'(s)=\bbf(\bx(s), \bu(s, \bx(s)))\hbox{ in }(t, T),\quad \bx(t)=\by.
$$
A given control mapping $\bu(s, \by)$ might have more than one associated state for a given pair $(t, \by)$ if Lipschitzianity is not enforced. 
However, to be able to treat  optimality conditions, we need some further regularity. 
Feasible control mappings will be taken from the space $L^2(0, T; H^1(\Omega; \bK))$ in order to ensure that spatial derivatives can be calculated. 

\begin{definition}\label{principal}
An optimal solution $\bU(t, \bx)$ of this optimization problem is called an optimal feedback law (for that same problem). 
\end{definition}

We would like to avoid the constraint $\bx(s)\in\Omega$ to skip further difficulties, and so we will deal instead with the optimization problem
$$
\hbox{Minimize in }\bu(t, \bx)\in\bK:\quad 
\int_D\int_0^T\int_t^TF(\bx(s), \bu(s, \bx(s)))\,ds\,dt\,d\by+g(\bx(T))
$$
subject to
$$
\bx'(s)=\bbf(\bx(s), \bu(s, \bx(s)))\hbox{ in }(t, T),\quad \bx(t)=\by, 
$$
for each pair $(t, \by)\in[0, T]\times D$, where $D$ is just a domain of interest for initial conditions, and such that state trajectories $\bx(s)\in D$ for every solution of the state system.  Even better, we can take $D=\R^N$, and suppose that we can count on some suitable growth conditions on $F$ so that the resulting integrals, with an infinite domain of integration, are finite. Since these properties will not play a central role in our analysis, we do not specify them. 

There are three main sources of concern about this problem that we plan to treat successively:
\begin{enumerate}
\item existence of an optimal control map $\bU(s, \bx)$;
\item optimality conditions that such an optimal map $\bU(s, \bx)$ should comply with;
\item procedure to approximate its values.
\end{enumerate}
We will explore these three fundamental issues in three respective sections. 

\begin{enumerate}
\item Concerning existence of optimal strategy, we will focus, for each arbitrary pair $(t, \bx)$, on the problem
$$
\hbox{Minimize in }\bv(t)\in\bK:\quad \int_t^TF(\by(s), \bu(s))\,ds+g(\by(T)),
$$
subjected to
$$
\by'(s)=\bbf(\by(s), \bv(s))\hbox{ in }(t, T),\quad \by(t)=\bx.
$$
Let $\tilde\bv(s; t, \bx)$ be its optimal solution. 

\begin{theorem}\label{existencia}
Suppose the ingredients $F$, $\bbf$, and $\bK$ of the original optimal control problem  are such that 
there is a unique optimal strategy $\tilde\bv(s; t, \bx)$, as just indicated, for every pair $(t, \bx)$. Then
$$
\bU(t, \bx)=\tilde\bv(t; t, \bx)
$$
is the optimal feedback law for the control problem according to Definition \ref{principal}. 
\end{theorem}
\item 
The statement of optimality conditions introduces a linear first-order PDE system. It tries to understand the joint dependence of the optimal feedback map $\bU(t, \bx)$ on its variables $(t, \bx)$ through its being an optimal solution of the corresponding feedback optimization problem. 

Let the costate $\bp(t, \bx)$, associated with the control $\bu(t, \bx)$, be the solution of the problem
\begin{equation}\label{primerorden}
\bp_t(t, \bx)+\nabla\bp(t, \bx)\,\bbf(\bx, \bu(t, \bx))+\bp(t, \bx)\,\nabla[\bbf(\bx, \bu(t, \bx))]=\nabla [F(\bx, \bu(t, \bx))]
\end{equation}
in $[0, T]\times\R^N$, 
under the terminal time condition $\bp(T, \bx)=\nabla g(\bx)$. Notice that the solution of this system through characteristics takes us back to solving the state and costate systems. 

\begin{theorem}
Let $\bu(t, \bx)$ be a feasible field for the feedback problem for which the linear transport system \eqref{primerorden}
admits a solution $\bp(t, \bx):[0, T]\times\R^N\to\R^N$. If $\bu(t, \bx)$ turns out to realize the minimum of the hamiltonian
$F(\bx, \bv)-\bp(t, \bx)\bbf(\bx, \bv)$ in the control variable $\bv\in \bK$
$$
F(\bx, \bu(t, \bx))-\bp(t, \bx)\bbf(\bx, \bu(t, \bx))=\min_{\bv\in \bK}\{F(\bx, \bv)-\bp(t, \bx)\bbf(\bx, \bv)\},
$$
then  $\bu(t, \bx)$ is a local minimum for the feedback optimal control problem. 
\end{theorem}
\item Finally, the basis of an iterative approximation procedure stems from optimality, and it ties together the three topics occurring in the title: optimal feedback control, linear first-order systems of PDE, and obstacle problems. 
Set 
$$
\nabla I(\bu)(t, \bx)\equiv F_\bu(\bx, \bu(t, \bx))-\bp(t, \bx)\bbf_\bu(\bx, \bu(t, \bx)).
$$
The reason for this special notation comes from the fact (see below) that that particular combination of partial derivatives occur in the differentiation of the cost function $I(\bu)$ for our feedback problem. Our main result is intimately related to optimality.

\begin{theorem}\label{optimalidad}
Let $\bu(t, \bx)$ be a feasible map for the optimal control problem under a constraint set $\bK$, which is assumed to be compact and convex, and let $\bp(t, \bx)$ be its associated costate as just indicated. Then the solution $\bU(t, \bx)$ of the obstacle problem, for each fixed time $t\in[0, T]$, 
$$
\hbox{Minimize in }\bU(t, \bx)\in \bK:\quad \int_{\R^N}\left(\frac12|\nabla\bU(t, \bx)-\nabla\bu(t, \bx)|^2+\nabla I(\bu)(t, \bx)(\bU(t, \bx)-\bu(t, \bx))\right)\,dx
$$
is a descent direction for the optimal control at $\bu(t, \bx)$,  in an average sense
$$
\int_{\R^N}\nabla I(\bu)(t, \bx)(\bU(t, \bx)-\bu(t, \bx)) \,d\bx\le0,
$$
for all $t\in[0, T]$. If such $\bu(t, \bx)$ is indeed optimal for the control problem, then 
the solution of the obstacle problem is $\bu(t, \bx)$ itself for all $t\in[0, T]$.
\end{theorem}
This result is the basis of an iterative approximation procedure:
\begin{enumerate}
\item Initialization. Take any initial $\bu_0(t, \bx)\in\bK$.
\item Iterative scheme until convergence: if $\bu_j(t, \bx)$ is known, then 
\begin{enumerate}
\item Compute the costate $\bp_j(t, \bx)$ by solving the corresponding linear, first-order PDE system for $\bu=\bu_j$.
\item Set $\nabla I(\bu_j)(t, \bx)=F_\bu(\bx, \bu_j(t, \bx))-\bp_j(t, \bx)\bbf_\bu(\bx, \bu_j(t, \bx))$.
\item Solve the obstacle problem to determine $\bU_j(t, \bx)$.
\item Update $\bu_j$ to $\bu_j+\epsilon\bU_j$ for some small $\epsilon$.
\end{enumerate}
\end{enumerate}
In practice, solving the obstacle problem Step (b)(iii) may be avoided by simply taking $\bU_j(t, \bx)$ as the solution of the mathematical programming problem
$$
\nabla I(\bU_j)(t, \bx)=\min_{\bv\in\bK}\left(F_\bu(\bx, \bv)-\bp_j(t, \bx)\bbf_\bu(\bx, \bv)\right).
$$
Formally, however, this $\bU_j(t, \bx)$ might show a dependence on the variable $\bx$ too weak, as it will also depend upon the $\bx$-regularity of the costate $\bp_j$ solution of (\ref{primerorden}), for an iterative procedure to be implementable (see Section \ref{optfeed} below). This is the main reason to consider the obstacle problem in order to ensure an improved $\bx$-regularity, as required by feasibility.
\end{enumerate}

We end up by examining briefly the typical LQR problem under this perspective to see how the classical Ricatti equation is recovered. There is hardly any additional example that can be treated explicitly, so that the numerical approximation becomes crucial. 
The immediate future asks, then, for testing this viewpoint in concrete examples starting with simple academic situations and proceeding with more and more elaborate problems. We are already working on that (\cite{fontpedregal}). 

\section{Existence result}
This section treats the proof of Theorem \ref{existencia}. It refers to the existence of optimal solutions of the problem
$$
\hbox{Minimize in }\bu(t, \bx)\in\bK:\quad 
\int_\Omega\int_0^T\int_t^TF(\bx(s), \bu(s, \bx(s)))\,ds\,dt\,d\by+g(\bx(T))
$$
subject to
$$
\bx'(s)=\bbf(\bx(s), \bu(s, \bx(s)))\hbox{ in }(t, T),\quad \bx(t)=\by, 
$$
for each pair $(t, \by)\in[0, T]\times\R^N$. We will identify this problem as the feedback form of the underlying optimal control problem. 

Let us look at the inner problem
$$
\hbox{Minimize in }\bv(t)\in\bK:\quad \int_t^TF(\by(s), \bv(s))\,ds+g(\by(T)),
$$
subjected to
$$
\by'(s)=\bbf(\by(s), \bv(s))\hbox{ in }(t, T),\quad \by(t)=\bx.
$$
Let $\tilde\bv(s; t, \bx)$ be its optimal solution. We are assuming, without specifying any particular situation, that the main ingredients $(\bK, F, g, \bbf)$ of the control problem enable existence of a unique optimal solution $\tilde\bv(s; t, \bx)$ for every pair $(t, \bx)$. 

The very nature of this optimal solution $\tilde\bv(s; t, \bx)$ is such that 
$$
\tilde\bv(r; s, \tilde\by(s; t, \bx))=\tilde\bv(r; t, \bx),\quad r\in[s, T],
$$
if $\tilde\by(s; t, \bx)$ is the (optimal) trajectory associated with $\tilde\bv(s; t, \bx)$, namely
\begin{equation}\label{ley}
\frac d{ds}\tilde\by(s; t, \bx)=\bbf(\tilde\by(s; t, \bx), \tilde\bv(s; t, \bx)).
\end{equation}
In particular
\begin{equation}\label{flujo}
\tilde\bv(s; s, \tilde\by(s; t, \bx))=\tilde\bv(s; t, \bx).
\end{equation}
Set
\begin{equation}\label{feedb}
\bU(t, \bx)=\tilde\bv(t; t, \bx).
\end{equation}
Then \eqref{flujo} means
$$
\tilde\bv(s; t, \bx)=\tilde\bv(s; s, \tilde\by(s; t, \bx))=\bU(s, \tilde\by(s; t, \bx)),
$$
and  \eqref{ley} implies
$$
\frac d{ds}\tilde\by(s; t, \bx)=\bbf(\tilde\by(s; t, \bx), \bU(s, \tilde\by(s; t, \bx))).
$$
This identity clearly shows that the field $\bU(t, \bx)$ determined through \eqref{feedb} has a corresponding path $\tilde\by(s; t, \bx)$, through the state law for our feedback optimal control problem, which is the collection of optimal paths of the problem. It is therefore elementary to check that $\bU(t, \bx)$ is the optimal solution of the feedback problem, and it is therefore the optimal feedback law. 

\section{Optimality conditions for an optimal control in feedback form}\label{optfeed}
With the ingredients indicated at the end of the Introduction, we would like to prove the following optimality criterium. We start with the easier situation having no restriction set $\bK$ for the control. 

\subsection{No restriction set for controls}
We focus on the optimal feedback control problem yielding the optimal feedback mapping
$$
\hbox{Minimize in }\bu(t, \bx)\in L^2(0, T; H^1(\R^N; \R^N)):\quad \int_t^TF(\bx(s), \bu(s, \bx(s)))\,ds
$$
subject to 
$$
\bx'(s)=\bbf(\bx(s), \bu(s, \bx(s)))\hbox{ a.e. in }(t, T),\quad \bx(t)=\by,
$$
for $(t, \by)\in[0, T]\times\R^N$ given. As indicated earlier in the Introduction, the solution of the costate system
\begin{equation}\label{systco}
\bp_t(t, \bx)+\nabla\bp(t, \bx)\,\bbf(\bx, \bu(t, \bx))+\bp(t, \bx)\nabla_\bx[\bbf(\bx, \bu(t, \bx))]
= \nabla_\bx [F(\bx, \bu(t, \bx))],
\end{equation}
under the terminal time condition $\bp(T, \bx)=0$, will play, as usual, a prominent role.  

\begin{theorem}\label{prin}
Let $\bu(t, \bx)$ be a feasible field for which the linear transport system \eqref{systco}
under the terminal time condition $\bp(T, \bx)=0$ admits a solution $\bp(t, \bx):[0, T]\times\R^N\to\R^N$. If 
\begin{equation}\label{optimalidad}
F_\bu(\bx, \bu(t, \bx))-\bp(t, \bx)\bbf_\bu(\bx, \bu(t, \bx))\equiv0,
\end{equation} 
then $\bu(t, \bx)$ is an equilibrium mapping for the feedback optimal control problem. 
\end{theorem}
The proof amounts to redoing the usual calculations in the classic context with the costate, performed in this feedback scenario.
\begin{proof}
Let $\bu(t, \bx)$ be feasible, and $\bx(t)$ be (one of) its associated state(s) so that
$$
\bx'(s)=\bbf(\bx(s), \bu(s, \bx(s))\hbox{ a.e. in }(t, T),\quad \bx(t)=\by.
$$
Let $\bU(t, \bx)$ be a feasible variation of $\bu(t, \bx)$, and write $\bX(t)$ for the variation produced on $\bx$ by $\bU$ on $\bu$. Then
$$
\bx'(s)+\epsilon\bX'(s)=\bbf(\bx(s)+\epsilon\bX(s), 
\bu(s, \bx(s)+\epsilon\bX(s))+\epsilon\bU(s, \bx(s)+\epsilon\bX(s)))\hbox{ in }(t, T),\quad \bX(t)=0.
$$
By differentiation with respect to $\epsilon$, and setting $\epsilon=0$ afterwards, we should have
\begin{equation}\label{optuno}
\bX'=\left(\bbf_\bx+\bbf_\bu\nabla\bu\right)\bX+\bbf_\bu\bU\hbox{ in }(t, T),\quad \bX(t)=0,
\end{equation}
where $\bbf_\bx$, $\bbf_\bu$ are evaluated at $(\bx(s), \bu(s, \bx(s)))$, and $\nabla\bu$ and $\bU$ are evaluated at $(s, \bx(s))$. Going over the same kind of calculations for the cost functional, we arrive at
\begin{equation}\label{optdos}
\int_t^T\left[\left(F_\bx+F_\bu\nabla\bu\right)\bX+F_\bu\bU\right]\,ds.
\end{equation}
Suppose that the field $\bu(s, \bx)$ is such that the conditions on the statement holds for $\bp(s, \bx)$. Then, it is clear that if we put $\bp(s)=\bp(s, \bx(s))$ for $\bx(s)$ the corresponding state,
\begin{gather}
\bp'(s)+\bp(s)[\bbf_\bx(\bx(s), \bu(s, \bx(s)))+\bbf_\bu(\bx(s), \bu(s, \bx(s)))\nabla\bu(s, \bx(s))]\label{coestadoo}\\
= F_\bx(\bx(s), \bu(s, \bx(s)))+F_\bu(\bx(s), \bu(s, \bx(s)))\nabla\bu(s, \bx(s)),\quad s\in[0, T],\nonumber
\end{gather}
with $\bp(T)=0$. If we take this information back to (\ref{optdos}), it is straightforward to get 
$$
\int_t^T\left[\left(\bp'(s)+\bp(s)(\bbf_\bx+\bbf_\bu\nabla\bu)\right)\bX+F_\bu\bU\right]\,ds.
$$
Integrating by parts in the first term, and bearing in mind that the boundary terms drop out, we obtain
$$
\int_t^T\left[-\bp\bX'+\bp(\bbf_\bx+\bbf_\bu\nabla\bu)\bX+F_\bu\bU\right]\,ds.
$$
Taking into account (\ref{optuno}), we can also write
$$
\int_t^T(-\bp\bbf_\bu\bU+F_\bu\bU)\,dt,
$$
which vanishes. The arbitrariness of $\bU$, $t$, and $\by$ finishes the proof.
\end{proof}

\subsection{Controls under constraints}
Constraints on the control variable through a set $\bK\subset\R^m$ can be easily incorporated. We assume that $\bK$ is compact and convex. 
We need to take into account that variations are now of the form $\bu+\epsilon(\bU-\bu)$ for arbitrary, feasible $\bU$, and so optimality conditions are one-sided conditions. This leads naturally to variational inequalities and obstacle problems. 
\begin{theorem}
Let $\bu(t, \bx)$ be as in Theorem \ref{prin}, and taking values on a convex, compact set $\bK$, for which $\bp(t, \bx):[0, T]\times\R^N\to\R^N$ is a solution of the problem
$$
\bp_t(t, \bx)+\nabla\bp(t, \bx)\,\bbf(\bx, \bu(t, \bx))+\bp(t, \bx)\nabla_\bx[\bbf(\bx, \bu(t, \bx))]
= \nabla_\bx [F(\bx, \bu(t, \bx))],
$$
with the terminal time condition $\bp(T, \bx)=0$. If $\bu(t, \bx)$ turns out to realize the minimum of the hamiltonian
$F(\bx, \bv)-\bp(t, \bx)\bbf(\bx, \bv)$ in the control variable $\bv\in \bK$
$$
F(\bx, \bu(t, \bx))-\bp(t, \bx)\bbf(\bx, \bu(t, \bx))=\min_{\bv\in \bK}\{F(\bx, \bv)-\bp(t, \bx)\bbf(\bx, \bv)\},
$$
then  $\bu(t, \bx)$ is a local minimum for the feedback optimal control problem. 
\end{theorem}
\begin{proof}
Note that the condition on the minimum implies that
$$
\left[F_\bv(\bx, \bu(t, \bx))-\bp(t, \bx)\bbf_\bv(\bx, \bu(t, \bx))\right](\bv-\bu(t, \bx))\ge0,
$$
for all $\bv\in \bK$. In particular, by choosing $\bv=\bU(t, \bx)$, we would have
$$
\left[F_\bv(\bx, \bu(t, \bx))-\bp(t, \bx)\bbf_\bv(\bx, \bu(t, \bx))\right](\bU(t, \bx)-\bu(t, \bx))\ge0.
$$
This implies that the local change on the cost functional for the variation $(\bU(t, \bx)-\bu(t, \bx))$ which is given, as above, by
$$
\int_t^T[-\bp\bbf_\bu(\bU-\bu)+F_\bu(\bU-\bu)]\,dt
$$
is non-negative. The arbitrariness of $\bU$, $t$ and $\by$ in $\bK$ yields the result. 
\end{proof}

\section{Approximation}
The two previous optimality results yield, when appropriately interpreted, an iterative approximation procedure for equilibrium mappings based on a typical steepest descent scheme with respect to a norm ensuring differentiability with respect to $\bx$. This differentiability issue makes the direction found a descent direction in the average with respect to the spatial variable $\bx$. 

Let $I(\bu)$ be the cost functional for the optimal control problem in feedback form. The computation in the proof of Theorem \ref{prin} shows that the field
$$
\nabla I(\bu)(t, \bx)\equiv F_{\bu}(\bx, \bu(t, \bx))-\bp(t, \bx)\bbf_{\bu}(\bx, \bu(t, \bx))
$$ 
regarded as a mapping of $(t, \bx)\in[0, T]\times\R^N$ represents the derivative 
$$
\left.\frac d{d\epsilon}I(\bu+\epsilon\bU)\right|_{\epsilon=0}
$$
in the sense that this derivative is actually the integral
$$
\int_t^T \nabla I(\bu)(s, \bx(s))\bU(s, \bx(s))\,ds
$$
for $\bx(s)$ the state associated with $\bu$, and initial condition $\by$ at time $t$. 

\begin{corollary}\label{media}
Let $\bu(t, \bx)$ be as in Theorem \ref{prin}, and determine the costate $\bp(t, \bx):[0, T]\times\R^N\to\R^N$ as a solution of the problem
$$
\bp_t(t, \bx)+\nabla\bp(t, \bx)\,\bbf(\bx, \bu(t, \bx))+\bp(t, \bx)\nabla_\bx[\bbf(\bx, \bu(t, \bx))]
= \nabla_\bx [F(\bx, \bu(t, \bx))],
$$
under the terminal time condition $\bp(T, \bx)=0$. Then the solution of the problem
$$
-\Delta\bU(t, \bx) +\nabla I(\bu)(t, \bx)=0\hbox{ in }\R^N
$$
for every $t\in[0, T]$, 
is a descent direction for the optimization problem at $\bu(t, \bx)$ in an average sense
$$
\int_{\R^N}\nabla I(\bu)(t, \bx)\bU(t, \bx)\,dx\le0
$$
for all $t\in[0, T]$. 
\end{corollary}
\begin{proof}
If we retake the computations in the proof of Theorem \ref{prin}, we find that the derivative of the cost functional of the optimal control in feedback form is given by the integral with respect to time $t$ and initial condition $\by$ of
$$
\int_t^T \nabla I(\bu)(s, \bx(s))\bU(s, \bx(s))\,ds,\quad \nabla I(\bu)(s, \bx)=F_{\bu}(\bx, \bu(t, \bx))-\bp(t, \bx)\bbf_{\bu}(\bx, \bu(t, \bx)),
$$
for every arbitrary perturbation $\bU(s, \bx)$. Suppose we take for this perturbation the unique solution $\bU(t, \bx)$ in $H^1(\R^N; \R^N)$ of the problem
$$
-\Delta\bU(t, \bx) +\nabla I(\bu)(t, \bx)=0\hbox{ in }\R^N
$$
for each time $t$. By using $\bU$ itself as a test function in this identity, an integration by parts leads immediately to
$$
\int_{\R^N}\nabla I(\bu)(t, \bx)\bU(t, \bx)\,dx=-\int_{\R^N}|\nabla\bU(t, \bx)|^2\,dx\le0,
$$
for each time $t$.
\end{proof}

Similar ideas for the restricted case in which we have a constraint set $\bK$, which we assume convex and compact, is to be respected can be used. 
Let $\bu(t, \bx)$ be as above,  under a constraint set $\bK$. Determine the costate $\bp(t, \bx):[0, T]\times\R^N\to\R^N$ as a solution of the problem
$$
\bp_t(t, \bx)+\nabla\bp(t, \bx)\,\bbf(\bx, \bu(t, \bx))+\bp(t, \bx)\nabla_\bx[\bbf(\bx, \bu(t, \bx))]
= \nabla_\bx [F(\bx, \bu(t, \bx))],
$$
under the terminal time condition $\bp(T, \bx)=0$. Recall that
$$
\nabla I(\bu)(t, \bx)=F_\bu(\bx, \bu(t, \bx))-\bp(t, \bx)\bbf_\bu(\bx, \bu(t, \bx)).
$$
This time we consider the obstacle problem, for each fixed time $t\in[0, T]$, 
\begin{equation}\label{obstaculo}
\hbox{Minimize in }\bU(t, \bx)\in \bK:\quad \int_{\R^N}\left(\frac12|\nabla\bU(t, \bx)-\nabla\bu(t, \bx)|^2+\nabla I(\bu)(t, \bx)(\bU(t, \bx)-\bu(t, \bx))\right)\,d\bx.
\end{equation}
It is standard to show that this problem has a unique solution $\bU(t, \bx)$ (see \cite{kinderstamp}). The underlying variational inequality  yields, in a straightforward way, that 
$$
\int_{\R^N}\nabla I(\bu)(t, \bx)(\bU(t, \bx)-\bu(t, \bx))\,d\bx\le -\int_{\R^N}|\nabla\bU(t, \bx)-\nabla\bu(t, \bx)|^2\le0
$$
for all $t$. This is Theorem \ref{optimalidad}. 

\section{The linear quadratic regulator}
Just as an illustration and as a confirmation of the perspective explained here, we look at the classical LQR situation, to check how the Ricatti equation arises in this context. One can hardly find a different situation where the previous formalism can be explicitly written.

In this case, we have the following ingredients:
\begin{enumerate}
\item $\bK$ and $\Omega$ are all of space, so that we do not have restrictions on state or control.
\item $F(\bx, \bu)$ is a quadratic integrand separately in both sets of variables
$$
F(\bx, \bu)=\frac12 \bx^*\bQ\bx+\frac12 \bu^*\bR\bu
$$
where $\bQ$ and $\bR$ are constant, symmetric, positive definite ($\bR$ strictly) matrices of the appropriate dimensions.
\item $\bbf(\bx, \bu)$ is linear in both sets of variables
$$
\bbf(\bx, \bu)=\bA\bx+\bB\bu
$$
for constant matrices $\bA$ and $\bB$ of the appropriate dimensions.
\item The cost functional typically incorporates a contribution involving the final state $\bx(T)$ in the form 
$$
\frac12 \bx(T)^*\bH\bx(T)
$$
with $\bH$, again, a symmetric, positive definite (not necessarily strictly) matrix. 
\end{enumerate}
With these ingredients, all we need to do is look at the costate equation
$$
\bp_t(t, \bx)+\nabla\bp(t, \bx)\,\bbf(\bx, \bu(t, \bx))+\bp(t, \bx)\nabla_\bx[\bbf(\bx, \bu(t, \bx))]
= \nabla_\bx [F(\bx, \bu(t, \bx))],
$$
under the terminal time condition $\bp(T, \bx)=0$, together with 
$$
F_\bu(\bx, \bu(t, \bx))-\bp(t, \bx)\bbf_\bu(\bx, \bu(t, \bx))=\cero.
$$
In the case for a LQR situation, these two pieces of information become
$$
\bp_t+\nabla\bp(\bA\bx+\bB\bu)+\bp(\bA+\bB+\nabla\bu)=\bQ+(\bR\bu)\nabla u
$$
with $\bp(T, \bx)=\bH\bx$, and
$$
\bR\bu-\bp\bB=\cero.
$$
Hence
$$
\bu=\bR^{1}\bB^*\bp,
$$
and substituting this information to eliminate $\bu$ from the equation for the costate $\bp$, we find that
\begin{equation}\label{coes}
\bp_t+(\nabla\bp\bA-\bQ)\bx+(\nabla\bp\bB\bR^{-1}\bB^*+\bA^*)\bp=\cero.
\end{equation}
The structure of this system for the unknown $\bp(t, \bx)$ clearly suggest that if we put
$$
\bp(t, \bx)=\bP(t)\bx,
$$
for a certain matrix-valued function $\bP(t)$, then system \eqref{coes} becomes
$$
\bP'(t)\bx+(\bP(t)\bA-\bQ)\bx+(\bP(t)\bB\bR^{-1}\bB^*+\bA^*)\bP(t)\bx=\cero.
$$
The terminal condition is $\bP(T)\bx=\bH\bx$. The arbitrariness of the variable $\bx\in\R^N$ leads to conclude that 
$$
\bP'(t)+\bP(t)\bA+\bA^*\bP(t)+\bP(t)\bB\bR^{-1}\bB^*\bP(t)-\bQ=\cero\hbox{ in }(0, T),\quad \bP(T)=\bH,
$$
the classical Ricatti equation for the matrix $\bP(t)$ relating state and costate at optimality.


\end{document}